\documentclass[11pt,reqno]{amsart}
\usepackage{graphicx}
\usepackage{verbatim}
\usepackage{textcomp}
\usepackage{amssymb}
\usepackage{cite}
\usepackage{amsmath}
\usepackage{latexsym}
\usepackage{amscd}
\usepackage{amsthm}
\usepackage{mathrsfs}
\usepackage{xypic}
\usepackage{bm}
\usepackage{url}
\usepackage{hyperref}

\newtheorem{thm}{Theorem}
\newtheorem{corr}[thm]{Corollary}
\newtheorem{lem}[thm]{Lemma}
\newtheorem{prop}[thm]{Proposition}

\theoremstyle{definition}

\newtheorem*{ack}{Acknowledgment}
\newtheorem{rem}[thm]{Remark}
\def\R{\mathbb R}

\def\N{\mathbb N}

\def\pt{\partial}

\begin{document}
\title[Comparison of Steklov eigenvalues and Laplacian eigenvalues]{Comparison of Steklov eigenvalues on a domain and Laplacian eigenvalues on its boundary in Riemannian manifolds}
\author{Changwei Xiong}
\address{Mathematical Sciences Institute, Australian National University, Canberra, ACT 2601, Australia}
\email{\href{mailto:changwei.xiong@anu.edu.au}{changwei.xiong@anu.edu.au}}
\date{\today}
\thanks{}

\maketitle

\begin{abstract}
We prove that in Riemannian manifolds the $k$-th Steklov eigenvalue on a domain and the square root of the $k$-th Laplacian eigenvalue on its boundary can be mutually controlled in terms of the maximum principal curvature of the boundary under sectional curvature conditions. As an application, we derive a Weyl-type upper bound for Steklov eigenvalues. A Pohozaev-type identity for harmonic functions on the domain and the min-max variational characterization of both eigenvalues are important ingredients.
\end{abstract}

\section{Introduction}\label{sec1}

Let $(M,g)$ be an $(n+1)$-dimensional Riemannian manifold and let $\Omega\subset M$ be a relatively compact domain with smooth boundary $\Sigma=\pt \Omega$. The Steklov eigenvalue problem, introduced by V.~A.~Steklov in 1895 (see \cite{KKK14}), is
\begin{equation}\label{eqS}
\begin{cases}
\Delta_\Omega u=0, & \text{ in }\Omega,\\
\dfrac{\pt u}{\pt\nu}=\sigma u, & \text{ on }\Sigma,
\end{cases}
\end{equation}
where $\nu$ is the outward unit normal along $\Sigma$. Equivalently, the Steklov eigenvalues form the spectrum of the Dirichlet-to-Neumann map $\Lambda:C^\infty(\Sigma)\rightarrow C^\infty(\Sigma)$ defined by
\begin{equation}
\Lambda f=\frac{\pt (Hf)}{\pt \nu},\: f\in C^\infty(\Sigma),
\end{equation}
where $Hf$ is the harmonic extension of $f$ to the interior of $\Omega$. The Dirichlet-to-Neumann map $\Lambda$ is a first order elliptic pseudodifferential operator \cite[pp. 37--38]{Tay96} and its spectrum is nonnegative, discrete and unbounded:
\begin{equation}
0=\sigma_0<\sigma_1\leq \sigma_2\leq \cdots \nearrow \infty.
\end{equation}
There is an extensive literature concerning the Steklov eigenvalue problem. We refer to the recent survey \cite{GP14} and the references therein for an account of this topic.

On the other hand, better-known is the Laplacian eigenvalue problem. Let $\Delta_\Sigma$ denote the Laplace-Beltrami operator acting on smooth functions on the boundary. Then the Laplacian eigenvalue problem is
\begin{equation}\label{eqL}
-\Delta_\Sigma f= \lambda f, \text{ on }\Sigma,
\end{equation}
and it admits an increasing discrete sequence of non-negative eigenvalues
\begin{equation}
0=\lambda_0<\lambda_1\leq \lambda_2\leq \cdots \nearrow \infty.
\end{equation}
It is well known that the principal symbol of the Dirichlet-to-Neumann map $\Lambda$ is the square root of the principal symbol of the Laplacian $\Delta_\Sigma$. See e.g. \cite[p. 38 and p. 453]{Tay96} and \cite{Sha71}. Consequently, we have
\begin{equation}
\sigma_j\sim \sqrt{\lambda_j},\text{ as }j\rightarrow \infty.
\end{equation}
Recently, Luigi~Provenzano and Joachim~Stubbe \cite{PS16} confirmed this phenomenon explicitly for a $C^2$ domain $\Omega$ in Euclidean spaces. More precisely, they proved that $|\sigma_j-\sqrt{\lambda_j}|$ can be controlled in terms of the geometry of the domain. Our purpose in the present paper is to investigate the same problem for domains in Riemannian manifolds. Our main result can be stated as follows.

\begin{thm}\label{thm1}
Let $(M^{n+1},g)$ be an $(n+1)$-dimensional complete Riemannian manifold. Denote by $K_M$ its sectional curvature. Let $\Omega\subset M^{n+1}$ be a bounded domain with boundary $\Sigma=\pt \Omega$ of class $C^2$. Denote by $II$ the second fundamental form of $\Sigma$.
\begin{enumerate}
  \item If $-a\leq K_M\leq 0$ and $0<\sqrt{a}\leq  II\leq \kappa_+$, then
  \begin{equation}
\lambda_j\leq \sigma_j^2+ n \kappa_{+}\sigma_j,\quad \sigma_j\leq \frac{\kappa_{+}}{2}+\sqrt{\frac{\kappa_{+}^2}{4}+\lambda_j},\quad j\in \N.
\end{equation}
In particular,
\begin{equation}
|\sigma_j-\sqrt{\lambda_j}|\leq \max\{n/2,1\}\kappa_+.
\end{equation}
  \item If $0<  K_M\leq a$ and $0\leq II\leq \kappa_+$, then
  \begin{equation}
\lambda_j\leq \sigma_j^2+ n\sqrt{a+ \kappa_+^2}\sigma_j,\quad \sigma_j\leq \frac{\sqrt{a+ \kappa_+^2}}{2}+\sqrt{\frac{a+ \kappa_+^2}{4}+\lambda_j},\quad j\in\N.
\end{equation}
Likewise,
\begin{equation}
|\sigma_j-\sqrt{\lambda_j}|\leq \max\{n/2,1\}\sqrt{a+ \kappa_+^2}.
\end{equation}
\end{enumerate}

\end{thm}
\begin{rem}
Case (1) of Theorem \ref{thm1} includes the result in Euclidean spaces due to \cite{PS16}, and the one in hyperbolic spaces; while Case (2) includes the spherical result, which degenerates to the Euclidean case as $a\rightarrow 0^+$.
\end{rem}
\begin{rem}
In hyperbolic case, e.g. $K_M=-1$, the condition $II\geq 1$ is called ``horo-convex'', which is a natural convexity. See e.g. \cite{GWW14} where this kind of convexity is essentially required. It is also worth mentioning that, in space forms it is very likely to prove results for even non-convex domains, just as in \cite{PS16}. Here we present the results for convex (horo-convex) domains just for simplicity. And we work with domains in Riemannian manifolds rather than with Riemannian manifolds with boundary, just in order to keep parallel with \cite{PS16}. In addition, we note that under the conditions in Theorem \ref{thm1} the domain $\Omega$ has only one boundary component (see e.g. \cite{Ale77, GM69}).
\end{rem}
\begin{rem}
There are other types of comparison between the Steklov eigenvalue $\sigma_j$ and the Laplacian eigenvalue $\lambda_j$, see e.g. \cite{WX09,CEG11,Kar15,YY17}.
\end{rem}

Therefore any bound for $\lambda_j$ will imply a bound for $\sigma_j$. In particular, by the Weyl-type bound for $\lambda_j$ \cite{Bus79}, we obtain the following:
\begin{corr}\label{corr1}
Notations as in Theorem \ref{thm1}. If $-a\leq K_M\leq 0$ and $0<\sqrt{a}\leq II\leq \kappa_+$, then
\begin{equation}
\sigma_j\leq \kappa_++C_n \left(\frac{j}{|\Sigma|}\right)^{\frac{1}{n}},\: j\in \N;
\end{equation}
if $0< K_M\leq a$ and $0\leq  II\leq \kappa_+$, then
\begin{equation}
\sigma_j\leq \sqrt{a+\kappa_+^2}+C_n \left(\frac{j}{|\Sigma|}\right)^{\frac{1}{n}},\: j\in \N,
\end{equation}
where $C_n$ is a constant depending only on $n$.
\end{corr}
Recall the well-known Weyl asymptotic formula (see e.g. \cite{GP14})
\begin{equation}\label{eq-W}
\sigma_j=2\pi \left(\frac{j}{\omega_n|\Sigma|}\right)^{\frac{1}{n}}+O(1),\text{ as }j\rightarrow \infty,
\end{equation}
where $\omega_n$ is the volume of the $n$-dimensional Euclidean unit ball. In view of \eqref{eq-W}, the power $1/n$ in Corollary \ref{corr1} is optimal.

The proof of Theorem \ref{thm1} follows Provenzano and Stubbe's work \cite{PS16}. First we prove a Pohozaev-type identity for a harmonic function $u$ on $\Omega$ by integrating $\Delta_\Omega u\cdot \langle F,\nabla u\rangle=0$ over $\Omega$, where $F$ is any Lipschitz vector field on $\Omega$. Then we choose a suitable $F$ which is supported on a tubular neighbourhood of the boundary $\Sigma$, so as to relate the two boundary integrals $\int_\Sigma \left(\frac{\pt u}{\pt \nu}\right)^2 d\sigma$ and $\int_\Sigma |\nabla_\Sigma u|^2 d\sigma$. Finally the min-max characterization for both eigenvalues implies the required result.

The paper is built up as follows. In Section \ref{sec2} we fix some notations, construct a potential function $\eta(x)$ on the tubular neighbourhood of the boundary in terms of the distance to the boundary and estimate the eigenvalues of its Hessian $\nabla^2\eta$. Then in Section \ref{sec3} we establish for general Lipschitz vector fields a Pohozaev-type identity and choose $F=\nabla \eta(x)$ to obtain the equivalence of $\int_\Sigma \left(\frac{\pt u}{\pt \nu}\right)^2 d\sigma$ and $\int_\Sigma |\nabla_\Sigma u|^2 d\sigma$. The final Section \ref{sec4} contains the proofs of Theorem \ref{thm1} and Corollary \ref{corr1}.

\begin{ack}
The author would like to thank the referee for careful reading of the paper and for
the valuable suggestions and comments which made this paper better and more readable. He is also grateful to Professor Ben~Andrews for stimulating discussions and consistent help. This work was supported by a postdoctoral fellowship funded via ARC Laureate Fellowship FL150100126.
\end{ack}

\section{Preliminaries}\label{sec2}

Let $(M^{n+1},g)$ be an $(n+1)$-dimensional complete Riemannian manifold with Levi-Civita connection $\nabla$. The Riemannian curvature tensor $R$ is given by
\begin{equation}
R(X,Y)Z=-\nabla_X\nabla_YZ+\nabla_Y\nabla_XZ+\nabla_{[X,Y]}Z
\end{equation}
for any $X,Y,Z\in \mathfrak{X}(M)$. Let $p\in M$ and $u,v\in T_pM$ linearly independent. Then the sectional curvature of a two-plane $u\wedge v$ at $p$ is defined by
\begin{equation}
K_M(u\wedge v)=\frac{\langle R(u,v)u,v\rangle}{||u\wedge v||^2}=\frac{\langle R(u,v)u,v\rangle}{||u||^2||v||^2-\langle u,v\rangle^2}.
\end{equation}
Assume $\Omega\subset M$ is a domain with $C^2$ boundary $\Sigma=\pt \Omega$. For any $x\in \Sigma$, let $\nu(x)$ be the outward unit normal to $\Sigma$. Then the second fundamental form $II$ of $\Sigma$ at $x$ is defined by
\begin{equation}\label{eq-SSF}
II(X,Y):=\langle \nabla_X\nu,Y\rangle,\:X,Y\in T\Sigma.
\end{equation}
Denote by $\kappa_1(x),\dots,\kappa_n(x)$ the principal curvatures of $\Sigma$ at $x$. Then there exist $\kappa_-$ and $\kappa_+$ in $\R$ such that
\begin{align*}
\kappa_-&=\inf_{x\in \Sigma}\inf_{i=1,\dots, n} \kappa_i(x),\\
\kappa_+&=\sup_{x\in \Sigma}\sup_{i=1,\dots, n} \kappa_i(x),
\end{align*}
in which case we also write $\kappa_-\leq II\leq \kappa_+$ for short.

 For any $x\in \bar{\Omega}$, set
\begin{equation}
d_0(x):=dist(x,\Sigma).
\end{equation}
Then we define an $h$-tubular neighbourhood $\omega_h$ of $\Sigma$ as
\begin{equation}
\omega_h:=\{x\in \Omega: d_0(x)<h\}.
\end{equation}
Since $\Sigma$ is of class $C^2$, every point in $\omega_h$ has a unique nearest point on $\Sigma$, provided $h>0$ is sufficiently small. Let $\bar{h}$ be a real positive number to be chosen such that for any $h\in (0,\bar{h})$ any point in $\omega_h$ has a unique nearest point on $\Sigma$. In the following we always assume $h\in (0,\bar{h})$.

For the Hessian of the distance function $d_0$, we recall the following comparison result due to A.~Kasue \cite{Kas82,Kas84} (See also \cite[Theorem 1.2.2]{Wan14}).
\begin{lem}
For constants $k$, $\theta\in\R$, let
\begin{equation}
f(t):=
\begin{cases}
\cos \sqrt{k}t-\dfrac{\theta}{\sqrt{k}}\sin \sqrt{k}t,&\text{ if } k> 0,\\
1-\theta t,&\text{ if }k=0,\\
\cosh \sqrt{-k}t-\dfrac{\theta}{\sqrt{-k}}\sinh \sqrt{-k}t,&\text{ if }k<0,
\end{cases}
\quad t\geq 0.
\end{equation}
Let $f^{-1}(0)\in (0,\infty]$ be the first zero point of $f$ and $h^+$ be the supremum of the width of the tubular neighbourhood in which $d_0$ is smooth.
\begin{enumerate}
 \item If $K_M\leq k$ and $II\leq \theta$, then for any $x\in \Omega$ with $d_0(x)<\min\{h^+,f^{-1}(0)\}$ and any unit $X\in T_xM$ orthogonal to $\nabla d_0(x)$,
 \begin{equation}
 \nabla^2 d_0(X,X)\geq \frac{f'}{f}(d_0(x)).
 \end{equation}
  \item If $K_M\geq k$ and $II\geq \theta$, then for any $x\in \Omega$ with $d_0(x)<\min\{h^+,f^{-1}(0)\}$ and any unit $X\in T_xM$ orthogonal to $\nabla d_0(x)$,
 \begin{equation}
 \nabla^2 d_0(X,X)\leq \frac{f'}{f}(d_0(x)).
 \end{equation}
 \end{enumerate}
\end{lem}

Denote the parallel hypersurface of $\Sigma$ with distance $h$ by
\begin{equation}
\Sigma_h=\pt \omega_h\setminus \Sigma.
\end{equation}
Define
\begin{equation}
d(x):=dist(x,\Sigma_h).
\end{equation}
So $d(x)=h-d_0(x)$, $\nabla d(x)=-\nabla d_0(x)$ and $\nabla^2 d(x)=-\nabla^2 d_0(x)$. Moreover, define
\begin{equation}
\eta(x):=
\begin{cases}
\dfrac{1}{2}d(x)^2,&\text{ if }K_M\leq 0,\\
1-\cos \sqrt{a}d(x),&\text{ if }0<K_M\leq a.
\end{cases}
\end{equation}
\begin{rem}
The definition of $\eta(x)$ in Euclidean case is the same as in \cite{PS16}; while in the case $0<K_M\leq a$ it is chosen such that $\nabla \eta(x)$ is a conformal vector field for a geodesic ball $\Omega$ in spheres, which is inspired by \cite{LWX14}. For these two model cases, $\nabla \eta$ being the conformal vector field leads to $\nabla^2 \eta=c(x)g$ for some smooth function $c(x)$, which is useful in simplifying the analysis.
\end{rem}

Let $\{\rho_i(x)\}_{i=1}^{n+1}$ be the eigenvalues of $\nabla^2 \eta(x)$. Assume that $\rho_1(x)\leq \rho_2(x)\leq \dots\leq \rho_{n+1}(x)$.
Then we can estimate these eigenvalues as follows.
\begin{lem}
Notations as above.
\begin{enumerate}
\item If $-a\leq K_M\leq 0$ and $0<\sqrt{a}\leq \kappa_-\leq II\leq \kappa_+$, then  the eigenvalues of $\nabla^2\eta(x)$ satisfy
\begin{equation}
0\leq \rho_i(x)\leq 1,\:1\leq i\leq n;\:\rho_{n+1}(x)=1.
\end{equation}
\item If $0<  K_M\leq a$ and $0\leq \kappa_-\leq II\leq \kappa_+$, then the eigenvalues of $\nabla^2\eta(x)$ satisfy
\begin{equation}
0\leq \rho_i(x)\leq a\cos \sqrt{a}d(x),\:1\leq i\leq n;\:\rho_{n+1}(x)=a\cos \sqrt{a}d(x).
\end{equation}
\end{enumerate}
\end{lem}
\begin{proof}
(1) First we notice that for any $X,Y\in T_xM$,
\begin{equation*}
\nabla^2 \eta(x)(X,Y)=\langle \nabla d,X\rangle \langle \nabla d,Y\rangle+d(x)\nabla^2 d(X,Y).
\end{equation*}
Then there is an eigenvalue $\rho_{n+1}(x)=1$ corresponding to the direction $\nabla d$. Assume that $\{E_i\}_{i=1}^n$ of unit length are the directions corresponding to $\{\rho_i(x)\}_{i=1}^n$. Then for any $E_i$:
\begin{equation}
\rho_i(x)=\nabla^2 \eta(x)(E_i,E_i)\leq -d(x)\frac{-\kappa_+}{1-\kappa_+ d_0(x)}.
\end{equation}
Note that by \cite[Theorem 3.11]{DL91} we can choose $\bar{h}=\kappa_+^{-1}>h$. Thus we have
\begin{equation}
\rho_i(x)\leq \frac{d(x)}{\kappa_+^{-1}- d_0(x)}= \frac{d(x)}{\bar{h}- d_0(x)}\leq 1.
\end{equation}

Similarly, we have
\begin{align*}
\rho_i(x)=\nabla^2 \eta(x)(E_i,E_i)&\geq -d(x)\frac{\sqrt{a}\sinh \sqrt{a}d_0-\kappa_-\cosh \sqrt{a}d_0}{\cosh \sqrt{a}d_0-\dfrac{\kappa_-}{\sqrt{a}}\sinh \sqrt{a}d_0}\\
&\geq \sqrt{a}d(x)\geq 0.
\end{align*}

(2) In this case we notice that for any $X,Y\in T_xM$,
\begin{equation*}
\nabla^2 \eta(x)(X,Y)=a\cos \sqrt{a}d(x)\langle \nabla d,X\rangle \langle \nabla d,Y\rangle+\sqrt{a}\sin \sqrt{a}d(x)\nabla^2 d(X,Y).
\end{equation*}
Then there is an eigenvalue $\rho_{n+1}(x)=a\cos \sqrt{a}d(x)$ corresponding to the direction $\nabla d$. Assume that $\{E_i\}_{i=1}^n$ of unit length are the directions corresponding to $\{\rho_i(x)\}_{i=1}^n$. Then for any $E_i$:
\begin{equation}
\rho_i(x)=\nabla^2 \eta(x)(E_i,E_i)\leq -\sqrt{a}\sin \sqrt{a}d(x)\frac{-\sqrt{a}\sin \sqrt{a}d_0-\kappa_+\cos \sqrt{a}d_0}{\cos \sqrt{a}d_0-\dfrac{\kappa_+}{\sqrt{a}}\sin \sqrt{a}d_0}.
\end{equation}
Note that by \cite[Theorems 3.11 and 3.22]{DL91} we can choose $\bar{h}$ such that $\tan(\sqrt{a}\bar{h})=\dfrac{\sqrt{a}}{\kappa_+}$. Therefore, we obtain:
\begin{equation}
\rho_i(x)\leq a\sin \sqrt{a}d(x)\frac{\dfrac{\sqrt{a}}{\kappa_+}\tan \sqrt{a}d_0+1}{\dfrac{\sqrt{a}}{\kappa_+}-\tan \sqrt{a}d_0}=\frac{a\sin \sqrt{a}d(x)}{\tan\sqrt{a}(\bar{h}-d_0(x))}\leq a\cos \sqrt{a}d(x).
\end{equation}
Similarly, we have
\begin{align*}
\rho_i(x)&\geq -\sqrt{a}\sin \sqrt{a}d(x)\frac{-\kappa_-}{1-\kappa_-d_0(x)}\geq 0.
\end{align*}

\end{proof}

\section{Pohozaev identity and its consequences}\label{sec3}

In this section we aim at proving the equivalence of two integrals $\int_\Sigma(\frac{\pt u}{\pt \nu})^2d\sigma$ and $\int_\Sigma |\nabla_\Sigma u|^2d\sigma$ for a harmonic function $u$ on $\Omega$. First we establish the following Pohozaev identity for $u$. The proof for it is similar to that in \cite{PS16}, except that here we need to take the covariant derivatives with respect to the connection $\nabla$.
\begin{lem}\label{lem9}
Let $F\in \Gamma(T\Omega)$ be a Lipschitz vector field. Let $u\in H^2(\Omega)$ with $\Delta u=0$ in $\Omega$. Then
\begin{align*}
\int_\Sigma \frac{\pt u}{\pt \nu}\langle F,\nabla u\rangle d\sigma&-\frac{1}{2}\int_\Sigma |\nabla u|^2 \langle F,\nu\rangle d\sigma\\
&+\frac{1}{2}\int_\Omega |\nabla u|^2\cdot div F dv-\int_\Omega \nabla F(\nabla u,\nabla u)dv=0.
\end{align*}
\end{lem}
Here and in the sequel $H^k(\Omega)$ denotes the standard Sobolev space $W^{k,2}(\Omega)$.

\begin{proof}
Since $u$ is harmonic, there holds $\Delta u\cdot\langle F,\nabla u\rangle=0$ in $\Omega$. Then we obtain
\begin{align}
0&=\int_\Omega \Delta u\cdot\langle F,\nabla u\rangle dv=\int_\Sigma \frac{\pt u}{\pt \nu}\langle F,\nabla u\rangle d\sigma-\int_\Omega \langle \nabla u,\nabla \langle F,\nabla u\rangle \rangle dv\nonumber \\
&=\int_\Sigma \frac{\pt u}{\pt \nu}\langle F,\nabla u\rangle d\sigma-\int_\Omega \nabla F(\nabla u,\nabla u) dv-\int_\Omega \nabla^2 u(F,\nabla u)dv.\label{eq1}
\end{align}
Now take $\{e_i\}_{i=1}^{n+1}$ as an orthonormal local frame for $T\Omega$. So we have
\begin{align*}
\nabla^2 u(F,\nabla u) &=  u_{ij}F_iu_j= (u_j F_i u_j)_i-u_jF_{i,i}u_j-u_jF_iu_{ji}\\
&=div(|\nabla u|^2 F)-|\nabla u|^2 \cdot div F-\nabla^2 u(F,\nabla u),
\end{align*}
which holds indeed globally. Then integration by parts yields
\begin{equation}
\int_\Omega \nabla^2 u(F,\nabla u)dv=\frac{1}{2}\int_\Sigma |\nabla u|^2 \langle F,\nu\rangle d\sigma-\frac{1}{2}\int_\Omega |\nabla u|^2\cdot div Fdv.\label{eq2}
\end{equation}
Plugging \eqref{eq2} into \eqref{eq1}, we complete the proof of the lemma.

\end{proof}

Now we choose
\begin{equation}
F(x):=\begin{cases}0,&\text{ if } x\in \Omega\setminus \omega_h,\\
\nabla \eta,&\text{ if } x\in \omega_h,\end{cases}
\end{equation}
where we recall that
\begin{equation}
\eta(x):=
\begin{cases}
\dfrac{1}{2}d(x)^2,&\text{ if }K_M\leq 0,\\
1-\cos \sqrt{a}d(x),&\text{ if }0<K_M\leq a.
\end{cases}
\end{equation}
Then $F$ is a Lipschitz vector field. If $K_M\leq 0$, we have $F(x)= h\cdot \nu(x)$ for $x\in \Sigma$, and then by Lemma \ref{lem9}
\begin{align*}
0&= h\int_\Sigma \left(\frac{\pt u}{\pt \nu}\right)^2 d\sigma- h\int_\Sigma |\nabla_\Sigma u|^2  d\sigma\\
&+\int_{\omega_h} (|\nabla u|^2 \Delta \eta -2\nabla^2 \eta(\nabla u,\nabla u))dv;
\end{align*}
while if $0<K_M\leq a$, we have $F(x)=\sqrt{a}\sin\sqrt{a} h\cdot \nu(x)$ for $x\in \Sigma$, and then again by Lemma \ref{lem9}
\begin{align*}
0&= \sqrt{a}\sin\sqrt{a} h\int_\Sigma \left(\frac{\pt u}{\pt \nu}\right)^2 d\sigma- \sqrt{a}\sin\sqrt{a} h\int_\Sigma |\nabla_\Sigma u|^2  d\sigma\\
&+\int_{\omega_h} (|\nabla u|^2 \Delta \eta -2\nabla^2 \eta(\nabla u,\nabla u))dv.
\end{align*}
In both cases we need to estimate the last term in the expressions, which is the content of the following lemma.
\begin{lem}\label{lem1}
Let $\Omega$ be a bounded domain in $M^{n+1}$ of class $C^2$ and $u\in H^1(\Omega)$.
\begin{enumerate}
  \item If $-a\leq K_M\leq 0$ and $0<\sqrt{a}\leq \kappa_-\leq II\leq \kappa_+$, then
  \begin{equation}
-\int_\Omega |\nabla u|^2 dv\leq \int_{\omega_h} (|\nabla u|^2 \Delta \eta -2\nabla^2 \eta(\nabla u,\nabla u))dv\leq n\int_\Omega |\nabla u|^2 dv.
\end{equation}
  \item If $0<  K_M\leq a$ and $0\leq \kappa_-\leq II\leq \kappa_+$, then
\begin{equation}
-a\int_\Omega |\nabla u|^2 dv\leq \int_{\omega_h} (|\nabla u|^2 \Delta \eta -2\nabla^2 \eta(\nabla u,\nabla u))dv\leq na\int_\Omega |\nabla u|^2 dv.
\end{equation}
\end{enumerate}
\end{lem}
\begin{proof}
In fact we will first prove a pointwise inequality. Then integrating it yields the result. Let $x\in \omega_h$. Denote by $\xi_i(x)$, $i=1,\dots,n+1$, the normalized eigenvectors of $\nabla^2\eta(x)$ corresponding to the eigenvalues $\rho_i(x)$, $i=1,\dots,n+1$. Then we can decompose $\nabla u (x)$ as
\begin{equation}
\nabla u(x)=\sum_{i=1}^{n+1} \alpha_i(x)\xi_i(x).
\end{equation}
Then
\begin{align*}
Q&:=|\nabla u|^2 \Delta \eta -2\nabla^2 \eta(\nabla u,\nabla u)\\
 &=|\nabla u|^2 \sum_{i=1}^{n+1}\rho_i(x)-2\sum_{i=1}^{n+1}\rho_i(x)\alpha_i(x)^2.
\end{align*}
Assume $\nabla u(x)\neq 0$. We can normalize $\alpha_i(x)$ to get
\begin{equation}
\tilde{\alpha}_i(x):=\frac{\alpha_i(x)}{\sqrt{\sum_{i=1}^{n+1}\alpha_i(x)^2}}=\frac{\alpha_i(x)}{|\nabla u(x)|}.
\end{equation}
Therefore we obtain
\begin{equation}
Q=\sum_{i=1}^{n+1}\rho_i(x)(1-2\tilde{\alpha}_i(x)^2)|\nabla u(x)|^2.
\end{equation}
Then direct computation yields (recall $\rho_1(x)\leq \rho_2(x)\leq \dots\leq \rho_{n+1}(x)$)
\begin{equation}
\sum_{i=1}^{n}\rho_i(x)-\rho_{n+1}(x)\leq \sum_{i=1}^{n+1}\rho_i(x)(1-2\tilde{\alpha}_i(x)^2)\leq \sum_{i=2}^{n+1}\rho_i(x)-\rho_1(x).
\end{equation}

Now in Case (1), for a lower bound, we notice that
\begin{align*}
\sum_{i=1}^{n+1}\rho_i(x)(1-2\tilde{\alpha}_i(x)^2)&\geq -1;
\end{align*}
while for an upper bound, we have
\begin{align*}
\sum_{i=1}^{n+1}\rho_i(x)(1-2\tilde{\alpha}_i(x)^2)&\leq n.
\end{align*}
Consequently,
\begin{equation}
-|\nabla u|^2 \leq Q \leq n|\nabla u|^2.
\end{equation}
Then by integrating the inequality we finish the proof of Case (1).

Case (2) can be handled similarly, with further using $\cos \sqrt{a}d(x)\leq 1$. So we complete the proof of the lemma.

\end{proof}
%

In the following we only deal with the case $-a\leq K_M\leq 0$ and $0<\sqrt{a}\leq \kappa_-\leq II\leq \kappa_+$, since the other case is analogous. The following proposition shows that the two integrals $\int_\Sigma |\nabla_\Sigma u|^2 d\sigma$ and $\int_\Sigma (\frac{\pt u}{\pt \nu})^2 d\sigma$ are equivalent.
\begin{prop}\label{prop1}
Assume $-a\leq K_M\leq 0$ and $0<\sqrt{a}\leq \kappa_-\leq II\leq \kappa_+$. Let $u\in H^2(\Omega)$ satisfy $\Delta u=0$ in $\Omega$ and normalized such that $\int_\Sigma u^2d\sigma =1$. Then we have
\begin{equation}
\int_\Sigma |\nabla_\Sigma u|^2 d\sigma\leq \int_\Sigma \left(\frac{\pt u}{\pt \nu}\right)^2 d\sigma+n\kappa_+\left(\int_\Sigma \left(\frac{\pt u}{\pt \nu}\right)^2 d\sigma\right)^{\frac{1}{2}},
\end{equation}
and
\begin{equation}
\left(\int_\Sigma \left(\frac{\pt u}{\pt \nu}\right)^2 d\sigma\right)^{\frac{1}{2}}\leq \frac{\kappa_+}{2}+\sqrt{\frac{\kappa_+^2}{4}+\int_\Sigma |\nabla_\Sigma u|^2 d\sigma}.
\end{equation}
\end{prop}
\begin{proof}
For the first inequality, by Lemma \ref{lem1}, we have
\begin{align*}
\int_\Sigma |\nabla_\Sigma u|^2 d\sigma &= \int_\Sigma \left(\frac{\pt u}{\pt \nu}\right)^2 d\sigma+\frac{1}{ h} \left(\int_{\omega_h} (|\nabla u|^2 \Delta \eta -2\nabla^2 \eta(\nabla u,\nabla u))dv\right)\\
&\leq \int_\Sigma \left(\frac{\pt u}{\pt \nu}\right)^2 d\sigma+\frac{n}{h} \int_{\Omega} |\nabla u|^2dv\\
&=\int_\Sigma \left(\frac{\pt u}{\pt \nu}\right)^2 d\sigma+\frac{n}{h} \int_{\Sigma} u\frac{\pt u}{\pt\nu}d\sigma\\
&\leq \int_\Sigma \left(\frac{\pt u}{\pt \nu}\right)^2 d\sigma+\frac{n}{h} \left(\int_\Sigma \left(\frac{\pt u}{\pt \nu}\right)^2 d\sigma\right)^{\frac{1}{2}},
\end{align*}
where in the last step we have used the Cauchy-Schwarz inequality. Then letting $h\rightarrow \bar{h}=\kappa_+^{-1}$ we get the first inequality.

For the second one, we get
\begin{align*}
\int_\Sigma \left(\frac{\pt u}{\pt \nu}\right)^2 d\sigma &=\int_\Sigma |\nabla_\Sigma u|^2 d\sigma-\frac{1}{h} \left(\int_{\omega_h} (|\nabla u|^2 \Delta \eta -2\nabla^2 \eta(\nabla u,\nabla u))dv\right)\\
&\leq \int_\Sigma |\nabla_\Sigma u|^2 d\sigma+\frac{1}{h} \int_{\Omega} |\nabla u|^2dv\\
&=\int_\Sigma |\nabla_\Sigma u|^2 d\sigma+\frac{1}{h} \int_{\Sigma} u\frac{\pt u}{\pt\nu}d\sigma\\
&\leq \int_\Sigma |\nabla_\Sigma u|^2 d\sigma+\frac{1}{h} \left(\int_\Sigma \left(\frac{\pt u}{\pt \nu}\right)^2 d\sigma\right)^{\frac{1}{2}},
\end{align*}
where again the Cauchy-Schwarz inequality has been used. Solving $\left(\int_\Sigma \left(\frac{\pt u}{\pt \nu}\right)^2 d\sigma\right)^{\frac{1}{2}}$ from it we have
\begin{equation}
\left(\int_\Sigma \left(\frac{\pt u}{\pt \nu}\right)^2 d\sigma\right)^{\frac{1}{2}}\leq \frac{1}{2h}+\sqrt{\frac{1}{4h^2}+\int_\Sigma |\nabla_\Sigma u|^2 d\sigma}.
\end{equation}
Likewise letting $h\rightarrow \bar{h}=\kappa_+^{-1}$ we get the second inequality.

\end{proof}

\section{Proofs of main results}\label{sec4}

\begin{proof}[Proof of Theorem \ref{thm1}]
The proof mainly utilizes Proposition \ref{prop1} and the min-max variational characterizations of the eigenvalues of problems \eqref{eqS} and \eqref{eqL}, i.e.
\begin{equation}
\sigma_j=\inf_{\substack{V\subset H^1(\Omega),\\dim V=j+1}}\sup_{\substack{ 0\neq u\in V,\\ \int_\Sigma u^2d\sigma=1}}\int_\Omega |\nabla u|^2 dv,
\end{equation}
for all $j\geq 0$, and
\begin{equation}\label{VCL}
\lambda_j=\inf_{\substack{V\subset H^1(\Sigma),\\dim V=j+1}}\sup_{ \substack{0\neq u\in V,\\ \int_\Sigma u^2d\sigma=1}}\int_\Sigma |\nabla_\Sigma u|^2 d\sigma,
\end{equation}
for all $j\geq 0$. More precisely, take the case $-a\leq K_M\leq 0$ and $0<\sqrt{a}\leq II\leq \kappa_+$ for example. Assume that $\{u_k\}_{k=0}^\infty\subset H^1(\Omega)$ is the sequence of eigenfunctions of the Steklov problem \eqref{eqS} corresponding to the eigenvalues $\{\sigma_k\}_{k=0}^\infty$. Moreover assume that $\int_\Sigma u_ku_ld\sigma=\delta_{kl}$ for $k,l\geq 0$. Then for fixed $j\geq 0$, considering $V=span\{u_0,u_1,\dots,u_j\}$, by the min-max variational characterization \eqref{VCL}, we have
\begin{align*}
\lambda_j&\leq \sup_{\sum_{k=0}^{j}c_k^2=1}\int_\Sigma \bigg|\nabla_\Sigma \left(\sum_{k=0}^{j}c_ku_k\right)\bigg|^2d\sigma\\
&\leq \sup_{\sum_{k=0}^{j}c_k^2=1}\left(\int_\Sigma \left(\frac{\pt \left(\sum_{k=0}^{j}c_ku_k\right)}{\pt \nu}\right)^2 d\sigma+n\kappa_+\left(\int_\Sigma \left(\frac{\pt \left(\sum_{k=0}^{j}c_ku_k\right)}{\pt \nu}\right)^2 d\sigma\right)^{\frac{1}{2}}\right)\\
&=\sup_{\sum_{k=0}^{j}c_k^2=1}\left(\sum_{k=0}^{j}c_k^2\sigma_k^2+n\kappa_+\left(\sum_{k=0}^{j}c_k^2\sigma_k^2\right)^{\frac{1}{2}}\right)\\
&=\sigma_j^2+n\kappa_+\sigma_j.
\end{align*}
Here the second inequality is due to Proposition \ref{prop1}. The other inequality can be proved similarly. See also \cite{PS16} for more details. So we finish the proof.
\end{proof}
To prove Corollary \ref{corr1}, we recall the following Weyl-type estimate due to P.~Buser \cite{Bus79}. (See also \cite{Kor93} and \cite{Has11}.)
\begin{thm}[\cite{Bus79}]
Let $(\Sigma, g)$ be a compact Riemannian manifold without boundary of dimension $n$ such that $Ric_g(\Sigma)\geq -(n-1)\kappa^2$, $\kappa\geq 0$. Then
\begin{equation}
\lambda_j\leq \frac{(n-1)\kappa^2}{4}+c_n\left(\frac{j}{|\Sigma|}\right)^{\frac{2}{n}},
\end{equation}
where $c_n>0$ depends only on $n$.
\end{thm}
Now we are ready to prove Corollary \ref{corr1}.
\begin{proof}[Proof of Corollary \ref{corr1}]
By Gauss equation, for an orthonormal frame $\{e_i\}_{i=1}^n\subset T\Sigma$, the Riemannian curvature tensors $R^\Sigma$ on the boundary $\Sigma$ and $R$ on $M$ are related by 
\begin{equation}
R^\Sigma_{ijkl}=R_{ijkl}+II_{ik}II_{jl}-II_{il}II_{jk},
\end{equation}
where $II_{ij}$ is the component of the second fundamental form $II$ of $\Sigma\subset M$ with respect to $\{e_i\}_{i=1}^n$. See the definition \eqref{eq-SSF}. Then it is easy to see for both cases in Corollary \ref{corr1} we have $Ric_g(\Sigma)\geq 0$. So $\lambda_j\leq c_n\left(\dfrac{j}{|\Sigma|}\right)^{\frac{2}{n}}$. Then using Theorem \ref{thm1}, in the case $-a\leq K_M\leq 0$ and $0<\sqrt{a}\leq \kappa_-\leq II\leq \kappa_+$, we get
\begin{equation}
\sigma_j\leq \kappa_++\sqrt{\lambda_j}\leq \kappa_++c_n^\frac{1}{2} \left(\frac{j}{|\Sigma|}\right)^{\frac{1}{n}},
\end{equation}
which is as claimed. The other case can be dealt with similarly.
\end{proof}

\bibliographystyle{Plain}

\begin{thebibliography}{10}

\bibitem{Ale77} S.~Alexander, \emph{Locally convex hypersurfaces of negatively curved spaces}, Proc. Amer. Math. Soc. \textbf{64} (1977), no.~2, 321--325.

\bibitem{Bus79} Peter~Buser, \emph{Beispiele f\"{u}r $\lambda_1$ auf kompakten Mannigfaltigkeiten}, Math.~Z. \textbf{165} (1979), no.~2, 107--133.

\bibitem{CEG11} Bruno~Colbois, Ahmad~El~Soufi and Alexandre~Girouard, \emph{Isoperimetric control of the Steklov spectrum}, J.~Funct.~Anal. \textbf{261} (2011), no.~5, 1384--1399.

\bibitem{DL91} Harold~Donnelly and Jeffrey~Lee, \emph{Domains in Riemannian manifolds and inverse spectral geometry}, Pacific J. Math. \textbf{150} (1991), no.~1, 43--77.

\bibitem{GWW14} Yuxin~Ge, Guofang~Wang and Jie~Wu, \emph{Hyperbolic Alexandrov-Fenchel quermassintegral inequalities II}, J.~Differential Geom. \textbf{98} (2014), no.~2, 237--260.

\bibitem{GP14} A.~Girouard and I.~Polterovich, \emph{Spectral geometry of the Steklov problem}, J. Spectr. Theory \textbf{7} (2017), no.~2, 321--359.

\bibitem{GM69} Detlef~Gromoll and Wolfgang~Meyer, \emph{On complete open manifolds of positive curvature}, Ann. of Math. (2) \textbf{90} (1969), 75--90.

\bibitem{Has11} Asma~Hassannezhad, \emph{Conformal upper bounds for the eigenvalues of the Laplacian and Steklov problem}, J.~Funct.~Anal. \textbf{261} (2011), no.~12, 3419--3436.

\bibitem{LWX14} Haizhong~Li, Yong~Wei and Changwei~Xiong, \emph{A note on Weingarten hypersurfaces in the warped product manifold}, Internat. J. Math. \textbf{25} (2014), no. 14, 1450121, 13 pp.

\bibitem{Kar15} Mikhail~A.~Karpukhin, \emph{Bounds between Laplace and Steklov eigenvalues on nonnegatively curved manifolds}, Electron. Res. Announc. Math. Sci. \textbf{24} (2017), 100--109.

\bibitem{Kas82} Atsushi~Kasue, \emph{A Laplacian comparison theorem and function theoretic properties of a complete Riemannian manifold}, Japan. J. Math. (N.S.) \textbf{8} (1982), no. 2, 309--341.

\bibitem{Kas84} Atsushi~Kasue, \emph{Applications of Laplacian and Hessian comparison theorems}, Geometry of geodesics and related topics (Tokyo, 1982), 333--386, Adv. Stud. Pure Math., 3, North-Holland, Amsterdam, 1984.


\bibitem{Kor93} Nicholas~Korevaar, \emph{Upper bounds for eigenvalues of conformal metrics}, J.~ Differential Geom. \textbf{37} (1993), no.~1, 73--93.

\bibitem{KKK14} N.~Kuznetsov, T.~Kulczycki, M.~Kwa\'{s}nicki, A.~Nazarov, S.~Poborchi, I.~Polterovich and B.~Siudeja, \emph{The legacy of Vladimir Andreevich Steklov}, Notices Amer. Math. Soc. \textbf{61} (2014), no.~1, 9--22.

\bibitem{PS16} Luigi~Provenzano and Joachim~Stubbe, \emph{Weyl-type bounds for Steklov eigenvalues}, J. Spectr. Theory, to appear, arXiv:1611.00929.

\bibitem{Sha71} S.~E.~Shamma, \emph{Asymptotic behavior of Stekloff eigenvalues and eigenfunctions}, SIAM J.~Appl.~Math. \textbf{20} (1971), 482--490.



\bibitem{Tay96} M.~E.~Taylor, \emph{Partial differential equations. II}, Applied Mathematical Sciences, 116, Springer-Verlag, New York, 1996.

\bibitem{Wan14} Feng-Yu~Wang, \emph{Analysis for diffusion processes on Riemannian manifolds}, Advanced Series on Statistical Science \& Applied Probability, 18, World Scientific Publishing Co. Pte. Ltd., Hackensack, NJ, 2014. xii+379 pp.

\bibitem{WX09} Qiaoling~Wang and Changyu~Xia, \emph{Sharp bounds for the first non-zero Stekloff eigenvalues}, J. Funct. Anal. \textbf{257} (2009), no.~8, 2635--2644.

\bibitem{YY17} Liangwei~Yang and Chengjie~Yu, \emph{A higher dimensional generalization of Hersch-Payne-Schiffer inequality for Steklov eigenvalues}, J. Funct. Anal. \textbf{272} (2017), no.~10, 4122--4130.

\end{thebibliography}

\end{document}